\documentclass[english,10pt]{article}
\usepackage[english]{babel}
\usepackage{amsmath,amsthm,amssymb ,latexsym, enumerate, hyperref, color, makeidx, xcolor, graphicx, amsmath, amsfonts, amsthm,amscd, amsmath, amssymb, amstext}
\allowdisplaybreaks
\newtheorem{theorem}{Theorem}
\newtheorem{corollary}[theorem]{Corollary}

\newtheorem{lemma}[theorem]{Lemma}

\newtheorem{proposition}[theorem]{Proposition}
\newtheorem{remark}[theorem]{Remark}

\title{Directional derivatives and the central limit theorem on compact general one-dimensional lattices}

\author{Artur O. Lopes\thanks{Institute for Mathematics and Statistics at UFRGS - Brazil}\; and Victor Vargas\thanks{Center for Mathematics at FCUP -  Portugal}}

\begin{document}

\maketitle

\begin{abstract}
We will show the central limit theorem for the general one-dimensional lattice where the space of symbols is a compact metric space. We consider the CLT for Lipschitz-Gibbs probabilities and in the proof  we use several  properties of the Ruelle operator defined on our setting; this will require fixing an {\em a priori probability}. An important issue  in the proof of the CLT is the existence of  a certain second-order derivative, and this will follow from the analytic properties that will be  described in detail throughout the paper. As additional results of independent interest, we will also describe some explicit estimates of the first and second directional derivatives of some dynamical entities like entropy and pressure. For example:  given a fixed potential $f$, and a variable observable $\eta$ on the Kernel of the Ruelle operator $\mathcal{L}_f$,  we consider  the equilibrium probability $\mu_{f + t \,\eta}$ for $f + t \,\eta$. We estimate the values 
$ \frac{d}{dt} h (\mu_{f + t \,\eta})|_{t=0}$ and $ \frac{d^2}{dt^2} h (\mu_{f + t \,\eta})|_{t=0}$, where $h (\mu_{f + t \,\eta})$ is the entropy of $ \mu_{C + t \,\eta}$. For fixed $f$ we can find conditions that can indicate the $\eta$ attaining the maximal possible  value  of $ \frac{d}{dt} h (\mu_{f + t \,\eta})|_{t=0}$ (up to a natural normalization of $\eta)$, entirely in terms of elements on the kernel of $\mathcal{L}_f$. We also consider  directional derivatives of the eigenfunction.
\end{abstract}



\section{Introduction}

Our main goal in this paper is to show the central limit theorem on a class of  compact general one-dimensional lattices (see Section \ref{clt-sec}), which are defined as symbolic spaces where the set of symbols is a compact metric space $M$ and the dynamics is played by the shift map $T : M^\mathbb{N} \to M^\mathbb{N}$; this kind of models is sometimes called in the mathematical literature {\em  generalized one-dimensional $XY$ models} (see \cite{BCLMS} and  \cite{LMMS}). We obtain the CLT for Lipschitz-Gibbs probabilities on $\Omega  := M^\mathbb{N}$ from the spectral properties of a Ruelle operator (to be defined in Section \ref{pre-sec}). This will require exploring properties of  the  so-called Ruelle operator as presented in \cite{LMMS}, where it is necessary to consider an {\em a priori probability} on $M$.

In the classical case (where $M=\{1,2,..,d\}$), well-known results in this direction have been obtained by diverse authors on several classes of dynamical contexts (for instance in \cite{Lal1}). In \cite{Liv} it was proven a CLT via martingale differences in the context of a bijective dynamical system satisfying suitable conditions, even more, it was shown a proof characterizing the CLT through properties of decay of correlations when the dynamics is surjective (see also \cite{Var} for an interesting approach). In \cite{Den} central limit theorems were proven for a wide class of dynamical systems, among them, geodesic flows, piece-wise expanding maps on the interval, and irrational rotations. In \cite{DG} the authors proved a CLT using as a main tool the so-called transfer operators defined on a suitable skew product. 

A key element in the proof of the CLT that we present in this paper is the existence of a certain second-order derivative and this will follow from the analytic properties of the Ruelle operator that will be properly described later (see Section \ref{aro-sec}). As additional results of independent interest, in Theorems \ref{faef1}, \ref{faef2}, Proposition \ref{dent} and Theorems  \ref{ddef14}, \ref{rafrug}, \ref{saKL}, we will also describe some explicit estimates of the first and second derivatives for some dynamical entities which are somehow related to the problem.

The characterization of the Ruelle operator in the context of the so-called generalized one-dimensional $XY$ models has been widely studied in different works. For instance, in \cite{LMMS} questions related to  the spectral properties of this kind of operator were addressed; it was also shown the existence of equilibrium states, a variational principle, and the existence of calibrated sub-actions. In \cite{SV} the authors  studied the mentioned problems on a model where the allowed sequences are given in terms of a continuous map of admissibility (a kind of generalization of  the classical shifts of finite type). The spectral gap (and the analyticity following this property) obtained in  item (4) in Theorem 1 in \cite{SV} contemplates our setting here as a particular case.

This paper is organized as follows: in Section \ref{pre-sec}  are presented the main definitions to be used throughout the paper, among them, the one for the Ruelle operator which we are interested in this context. In Section \ref{aro-sec} we state some well-known properties about the analyticity of some observables derived from the Ruelle operator and how it depends on the so-called spectral gap property. In Section \ref{clt-sec} we present the proofs of the main results of the paper, among them a  generalized version of the CLT in the context of the so-called {\em general one-dimensional lattices}. At the end of the paper, in Section \ref{Seco}, we present  explicit calculations for the directional derivatives of  entropy and pressure in terms of elements on the Kernel of the Ruelle operator.

We can take advantage of knowing differentiability in the present setting to investigate questions related to the first and second-directional derivatives of variations on time of dynamical entities related to pressure and entropy.   As an example, considering a fixed potential $f$ and a  variable observable $g$, denote by $\mu_{f + tg}$ the Lipschitz-Gibbs  probability (which is also equilibrium for the pressure of $ f + tg$) for the perturbed potential $f + tg$, when  $t$ is close to zero. It is natural to ask about the values
\begin{equation} \label{kjhd} \frac{d}{d t} h (\mu_{f + tg})|_{t=0}\,\,\,\,\text{and}\,\,\,\, \frac{d^2}{d t^2} h (\mu_{f + tg})|_{t=0}
\end{equation}
where $h (\mu_{f + tg})$ is the entropy of $ \mu_{f + tg}$ (see Proposition \ref{vaga1} and Corollary \ref{cvaga1}).

We can ask about the direction $g$ where the   value $\frac{d}{d t} h (\mu_{f + tg})|_{t=0} $ is the highest possible and how to get analytic  estimates of such $g$; all of that in terms of elements in the Kernel of the Ruelle operator $\mathcal{L}_f$. This can indicate the direction $g$ of the maximal increasing of entropy. Questions of this nature and  others are the main topics of Section \ref{Seco}. The estimation of  $\frac{d}{d t} h (\mu_{f + tg})|_{t=0}$ was already considered in \cite{GKLM}, but for other results we consider here it is appropriate to  prove things in a different way. 

We also consider directional derivatives associated with the pressure via expressions based entirely  on elements on the kernel of the Ruelle operator (see Theorem \ref{Prepre}). For  the explicit expression \eqref{vaga333} we take advantage of  the existence of a certain orthogonal family on the Kernel of the Ruelle operator (see \cite{LR1}).

 We also consider  directional derivatives of the eigenfunction of the Ruelle operator (see 
 Theorems \ref{rafrug} and \ref{saKL}).

\section{Preliminaries}
\label{pre-sec}

Consider a {\em compact metric space} $(M, d_M)$ satisfying $\mathrm{diam}(M) = 1$ and define $\Omega := M^{\mathbb{N}}$ as the set of sequences taking values on $M$ with the {\em shift map} given by $T((x_n)_{n = 1}^\infty) := (x_{n+1})_{n =1}^\infty$ acting on it. It is well-known that $\Omega$ equipped with the product topology results in a compact metric space and, even more, it is metrizable with $d(x, y) := \sum_{n = 1}^\infty\frac{1}{2^n}d_M(x_n, y_n)$. 

Throughout the paper, we denote by $\mathcal{B}_\Omega$, resp. $\mathrm{C}(\Omega)$, resp. $\mathrm{Lip}(\Omega)$, resp. $\mathcal{M}(\Omega) := \mathrm{C}(\Omega)^*$ resp. $\mathcal{M}_1(\Omega)$, resp. $\mathcal{M}_T(\Omega)$; the {\em Borel sigma-algebra} on $\Omega$, resp. the space of {\em continuous functions} from $\Omega$ into $\mathbb{R}$, resp. {\em Lipschitz continuous functions} from $\Omega$ into $\mathbb{R}$, resp. {\em Borel finite measures} on $\Omega$, resp. {\em Borel probability measures} on $\Omega$, resp. {\em Borel $T$-invariant probability measures} on $\Omega$. 

In this work, we are interested in the statistical behavior of probability measures obtained as the fixed points of the dual of a Ruelle operator which we call Lipschitz-Gibbs probabilities. The {\em Ruelle operator} for a potential $f \in \mathrm{C}(\Omega)$ is defined as the linear one assigning to each $w \in \mathrm{Lip}(\Omega)$ the map $\mathcal{L}_f(w) \in \mathrm{Lip}(\Omega)$, given by the expression
\begin{equation}
\label{ruop}
\mathcal{L}_f(w)(x) := \int_M e^{f(ax)}w(ax) d\nu(a) \;,
\end{equation}
where $\nu \in \mathcal{M}_1(M)$ is an {\em a priori measure} (with support equal $M$), and  each sequence $ax := (a, x_1, x_2, ...) \in \Omega$ satisfies $T(ax) = x$.

By the {\em RPF theorem} (see \cite{LMMS}, or \cite{LTF}, for details), for each $f \in \mathrm{Lip}(\Omega)$, there are $\lambda_f > 0$, $w_f \in \mathrm{Lip}(\Omega)^+$ and $\rho_f \in \mathcal{M}_1(\Omega)$, such that, 

\begin{equation}\label{rpf1}\mathcal{L}_f(w_f) = \lambda_f w_f \text{ and } \mathcal{L}_f^*(\rho_f) = \lambda_f \rho_f \;,
\end{equation}
and, even more, for any $w \in \mathrm{Lip}(\Omega)$ the following limit holds

\begin{equation}\label{rpf2}\lim_{n \to \infty} \Bigl\| \frac{\mathcal{L}_f^n(w)}{\lambda_f^n} - w_f \int_\Omega w d\rho_f \Bigr\|_\infty = 0 \;.
\end{equation}
 
A straightforward argument shows that the limit in \eqref{rpf2} implies the spectral gap property, i.e., that the main eigenvalue $\lambda_f$ is maximal and isolated into the spectrum of the operator (see for details \cite{SV}).

A potential $f \in \mathrm{Lip}(\Omega)$ is called {\em normalized} when $\mathcal{L}_f(1) = 1$,  we denote the set of normalized potentials by $\mathcal{N}(\Omega)$. It is easy to check that for $f \in \mathcal{N}(\Omega)$ the operator $\mathcal{L}_f^*$ preserves $\mathcal{M}_1(\Omega)$ and it is well-known that for arbitrary $f \in \mathrm{Lip}(\Omega)$, the {\em normalization} $N(f) := f + \log(w_f) - \log(w_f \circ \sigma) - \log(\lambda_f)$ belongs to $\mathcal{N}(\Omega)$. 

We call here $P(f):=\log \lambda_f$ the {\em pressure} of $f:\Omega \to \mathbb{R}$ (which due to the classical variational formula, as in  \cite{PP} or \cite{LMMS}, will play a  role only in Section \ref{Seco}).

We denote by $\mu_f \in \mathcal{M}_1(\Omega)$ the {\em Lipschitz-Gibbs probability for $f$, which is the unique fixed point for the dual Ruelle operator associated to the normalization $N(f)$, i.e., $\mathcal{L}_{N(f)}^*(\mu_f) = \mu_f$.} Moreover, it is easy to check that $\mu_f \in \mathcal{M}_T(\Omega)$ and, even more, choosing the eigenfunction and eigenprobability in \eqref{rpf1} satisfying $\int_\Omega w_f d\rho_f = 1$, {we have $\frac{d\mu_f}{d\rho_f} = w_f$ (see \cite{LMMS} for details). Besides that, under suitable assumptions, the identification  map $f \mapsto \mu_f$ is bijective and analytic (see \cite{BCV}). In this way, the set $\mathcal{N}(\Omega)$ can be also identified as the set of Lipschitz-Gibbs probabilities associated with all possible Lipschitz potentials, which we also denote by $\mathcal{N}(\Omega)$.

In the following sections we elaborate on the analyticity of the pressure map and later we use that to prove a version of the central limit theorem. For considering the meaning  of equilibrium state for a given Lipschitz continuous potential, it is necessary the concept of entropy and pressure for $T$-invariant probabilities on the $XY$ model (see \cite{LMMS}). Actually, in the mentioned work, it is shown that the so-called Ruelle operator  (taking into account an  {\em a priori probability}) has several interesting properties allowing to characterize the dynamical and statistical behavior of the map $T$. 

For the benefit of the reader, in the first part of the next section we will state some  properties of Thermodynamic Formalism for the generalized $XY$ model, among them analyticity issues,  that we will need later. Our focus is on Lipschitz-Gibbs probabilities.

\section{Analyticity of the Ruelle Operator}
\label{aro-sec}

Our main goal in the next section is to prove a central limit theorem on the setting of one-dimensional lattices defined on compact metric spaces, using as a main tool the analyticity of the map $f \mapsto P(f)$, more specifically, using the existence of first and second derivatives  of the function $t \mapsto P(f + tg)$, with $t \in \mathbb{R}$. 

In order to do that, first we introduce some important analytic properties of functions related to the Ruelle operator defined in Section \ref{pre-sec} and its corresponding dual, besides their eigenfunctions and eigenprobabilities, the above as a  function of the potential $f$. We will state some useful properties of the Ruelle operator given by the expression in \eqref{ruop}.

It is known that for any $f \in \mathrm{Lip}(\Omega)$ and each $n \in \mathbb{N}$ the map $f \mapsto \mathcal{L}^n_f$ is analytic (see \cite{BCV, PP} for the classical case, and \cite{SV, SSS} for the $XY$ model case). Even more, the expression for the Fr\'echet derivative of $\mathcal{L}^n_{(\cdot)}$ at the point $f$ in the direction of $g$ is given by
\begin{equation}
D(\mathcal{L}^n_{(\cdot)})_f(g) = \sum_{i=1}^{n} \mathcal{L}^i_f(g \mathcal{L}^{n-i}_f(g)) \;.
\end{equation}

The former expression also helps to guarantee that both of the maps $f \mapsto w_f$ and  $f \mapsto \lambda_f = e^{P(f)}$ are analytic. In order to guarantee the former claim, we need to introduce a strong result requiring advanced technical tools in Functional Analysis and the so-called approximation theory (for instance the ones presented in Chapter VII Sections 3-6 in \cite{Co}, see also Section 5 in \cite{ManeTF} and Section IV in \cite{Lor}).  

A detailed study of the spectrum of the Ruelle operator appears in \cite{FaJi} and \cite{Bala} under the assumption that $f \in \mathrm{Lip}(\Omega)$. In the mentioned works it is shown that the space of eigenfunctions associated with the main eigenvalue has dimension 1 and, even more, by a general result in \cite{LMMS} the eigenvalue $\lambda_f$ results isolated. Using the same reasoning, one can take advantage  of the following propositions (see \cite{FaJi}, \cite{Man} and \cite{ManeTF}) to get similar results in our setting. 

\begin{proposition}{\cite{BCV,CV,ManeTF,Lor}} \label{espectro2} 
Given $f \in \mathrm{Lip}(\Omega)$,  not necessarily normalized,   consider a small circle curve $\gamma:[0,1] \to \mathbb{C}$, $\gamma(t) = \epsilon\,( e^{2 \pi i t} + \lambda_f)$ around the eigenvalue $\lambda_f$ (by the spectral gap property no other elements of the spectrum of $\mathcal{L}_f$ belong to the interior of the curve). Then, the eigenfunction $w_g$, where the map $g \in \mathrm{Lip}(\Omega)$ is close by $f$,   can be obtained as the  line  integral
\begin{equation} \label{rrer} w_g = \frac{1}{2 \pi\, i} \Bigl(\int_\gamma  (z \,I - \mathcal{L}_g)^{-1} dz\Bigr)(1)  \,.
\end{equation}

From the former expression, it follows that the map $g \mapsto w_g$ is analytic on a neighborhood of $f$ contained into the space $\mathrm{Lip}(\Omega)$.
\end{proposition}

Note that the eigenfunction $w_g$ for $\mathcal{L}_g$ described by \eqref{rrer} is not necessarily normalized (in the sense that $\int_\Omega w_g d\rho_g = 1$).  In fact, the {\em projection operator} $\pi_1(\mathcal{L}_g)$ given by the expression
\begin{equation} \label{rrer2} 
\pi_1 (\mathcal{L}_g) := \frac{1}{2 \pi\, i} \Bigl(\int_\gamma  (z \,I - \mathcal{L}_g)^{-1} dz\Bigr) \;,
\end{equation}
is describing the projection of the space $\mathrm{Lip}(\Omega)$ on the one-dimensional space of eigenfunctions for the operator $ \mathcal{L}_g$ (see for instance \cite{Co}).  Furthermore, the function  $g \mapsto \pi_1 (\mathcal{L}_g)$ is complex analytic and any $w \in \mathrm{Lip}(\Omega)$ satisfies 
$$\mathcal{L}_g ( \pi_1 (\mathcal{L}_g)(w)) = \lambda_g \pi_1 (\mathcal{L}_g)(w).$$ In the next proposition we use the operator in \eqref{rrer2} to guarantee analyticity of the maps $f \mapsto \lambda_f$ and $f \mapsto P(f)$. 

\begin{proposition}\label{espectro3} 
Let $f \in \mathrm{Lip}(\Omega)$ be a not normalized potential. Then, the map $g \mapsto \lambda_g$ is analytic on the variable $g$ varying on a neighborhood of $f$. By the above, it follows that also the pressure map $g \mapsto P(g) = \log (\lambda_g)$ is analytic on the same neighborhood of $f$. 
\end{proposition}

\begin{proof} Note that  $\pi_1 (\mathcal{L}_g) (w_f)$ is an eigenfunction for the operator $ \mathcal{L}_g$ associated to $\lambda_g$, and varies analytically on $g$. Consider  $\rho \in \mathcal{M}_1(\Omega)$ supported on $\Omega$, such that, $\int_\Omega w_f d \rho \neq 0$. Then,
\begin{equation} \label{rre9}  
\lambda_g = \frac{ \int_\Omega \mathcal{L}_g  (\pi_1 (\mathcal{L}_g) (w_f))\, d \rho }{\int_\Omega  \pi_1 (\mathcal{L}_g) (w_f)\, d \rho }.
\end{equation}

Note that the denominator is not zero for $g$ close to $f$, which guarantees that the former expression is well-defined and a quotient of analytic functions. Therefore, the function $g \mapsto \lambda_g$ results in analytic on the variable $g$.
\end{proof}

{Consider $\theta$ acting on $\mathrm{Lip}(\Omega)$ given by $\theta(f) := \mathcal{L}_{f}$. It is well-known that the map $\theta$ is analytic on all the space $\mathrm{Lip}(\Omega)$ (see \cite{Contre,Man, SSS} for details).

More precisely, considering a fixed $w \in \mathrm{Lip}(\Omega)$, we have that the Fr\'echet derivative of $\theta$ at $f$ in direction of $g$ satisfies the expression
\begin{align*} 
\theta(f + g)\,(w) -  \theta(f)(w) 
&= \sum_{j=1}^\infty\, \frac{1}{j!}  \theta(f)(w\, g^j) \\
&= \label{ruru1}\sum_{j=1}^\infty\, \frac{1}{j!} \mathcal{L}_f(w\, g^j)
= \sum_{j=1}^\infty\, \frac{1}{j!} D^j \theta (f)\,( \underbrace{g, g,...,g}_j)\,.
 \end{align*}
 
 \medskip
 
In \cite{SSS} the authors show analytic properties of the Ruelle operator defined in the expression \eqref{ruop}, even more, from Theorem 3.5 in \cite{SSS}, we have that for $f, g, w \in \mathrm{Lip}(\Omega)$ the following analytic expression holds
 \begin{equation} \label{po1} 
 \mathcal{L}_{f + g}(w)(x)  =\sum_{j=1}^\infty\, \frac{1}{j!}\, \int_M e^{f(ax) }\, w(ax) \,g(a x)^j d \nu(a) \;.
 \end{equation}
 
 In particular, a second-order expression for Taylor's polynomial in our setting is given by
 \begin{equation} \label{po2}
 \mathcal{L}_{f + g}(w)(x)  =\, \, \int_M e^{f(ax) }\, w(ax) \,\Bigl(1 + g(ax) + \frac{g(ax)^2}{2}\Bigr) d \nu(a) + o_3(g) \;.
 \end{equation}

It is important to point out that the proof of the former expression was obtained without using the spectral gap property of the Ruelle operator (well-known in the case where $M$ is a finite set as described in \cite{ManeTF} and \cite{PP}).
 
Above it is considered the strong norm of operators and that $\mathcal{M}(\Omega) := \mathrm{C}(\Omega)^*$ is the space of continuous linear functionals with domain on $\mathrm{C}(\Omega)$ (see note on top page 190 in \cite{Lal2} for an interesting remark).

It follows that if we fix $f \in \mathrm{Lip}(\Omega)$ and consider a family of potentials $tf$, where $t \in \mathbb{R}$, then, the maps $t \mapsto w_{tf}$ and $t \mapsto \lambda _{tf}$ are analytic on the variable $t$. For the corresponding result for the map $t \mapsto \rho_{tf}$, see Note on top  page 190 in \cite{Lal2}. In the next lemma, we present an explicit expression for the Gateaux derivative for the map $f \mapsto \lambda_f$.

\begin{lemma}{\cite{BCV, GKLM}} \label{ddef1}
Consider $f, g \in \mathrm{Lip}(\Omega)$, $t \in \mathbb{R}$, and assume that $f \in \mathcal{N}(\Omega)$. Then,
\begin{equation} \label{lulu4}  
\frac{d}{dt} \lambda_{f + tg}|_{t=0} =  \frac{d}{dt} e^{ P(f + tg)}|_{t=0} = \int_\Omega g \, d\mu_f \,.
\end{equation}

If the potential $f$ is not normalized. Then,
\begin{equation} \label{lulu41}\frac{d}{dt} \lambda_{f + tg}|_{t=0} = \frac{d}{dt} e^{ P(f + tg)}|_{s=0} = \lambda_f\, \int_\Omega g \, d\mu_f \,.
\end{equation}
\end{lemma}

The former lemma implies directly the following result characterizing the Gateaux derivative for the function $f \mapsto P(f)$.

\begin{lemma}{\cite{BCV, GKLM}} \label{late} 
Given $f, g \in \mathrm{Lip}(\Omega)$ and $t \in \mathbb{R}$, denote by $p(t) := P(f + tg)$ the pressure of the potential $f + tg$. Then, we have
\begin{equation} \label{lulu473} 
p^\prime(0) = \frac{d}{dt} P(f + tg)|_{t=0} = \int_\Omega g \, d\mu_f \;.
\end{equation}
\end{lemma}

\begin{remark} \label{lilim2}
Note that the possible values of $\frac{d}{dt} e^{p(t)}|_{t=s}$ are on a finite sub-interval $(a, b)\subset \mathbb{R}$. Moreover, we have
\begin{equation} \label{lulu415} 
p^\prime(s) = \frac{d}{dt} p(t)|_{t=s} = \int_\Omega g \, d\mu_{f + sg} \;.
\end{equation}
Actually, it is only necessary to take $r := t-s$, $\psi := f + sg$ and later calculate $\frac{d}{dr} p(r)|_{r=0}$ such as appears in Lemma \ref{late}.
\end{remark}

The next theorem is the main purpose of \cite{GKLM} and provides a structure of the Riemann manifold for $\mathcal{N}(\Omega)$.

\begin{theorem}{\cite{GKLM}} \label{faef7} 
The set  $\mathcal{N}(\Omega)$ is an analytic infinite dimensional manifold, and the tangent space of $\mathcal{N}(\Omega)$ at $f$ satisfies $T_f (\mathcal{N}(\Omega)) = \mathrm{Ker}(\mathcal{L}_f)$, with $\mathcal{L}_f$ acting on the space $\mathrm{Lip}(\Omega)$. Furthermore, the inner product $< v, w > := \int_\Omega v w d \mu_f$, where  $v, w  \in  T_f (\mathcal{N}(\Omega))$, provides a Riemann structure on the manifold $\mathcal{N}(\Omega)$. In particular, the norm of an element $\eta \in \mathrm{Ker}(\mathcal{L}_f)$ is given by $\|\eta\|=  \sqrt{\int_\Omega \eta^2  d\mu _f}$. 
\end{theorem}

We point out that the results of \cite{GKLM} are also applied to the present case, where we consider the action of the shift on $\Omega$ and when the space of symbols $M$ is a compact metric space. This is so because the main ingredient was the use of Ruelle operators (in the same way that can be handled here)  plus analyticity (which is true in our setting).

\smallskip

\begin{remark}
The calculation of the Gauss curvature for $v, w  \in  T_f (\mathcal{N}(\Omega))$ appears in \cite{LR1} and the existence of geodesics in this context is the topic of \cite{LR2}.
\end{remark}

\medskip

The Riemannian metric described in Theorem \ref{faef7}  is related (but somehow different) to the so-called Pressure Riemannian  metric (see \cite{BCS} and \cite{McM}).

In Section \ref{Seco} we will be interested, among other topics, in estimating directional derivatives in the direction $\eta \in \mathrm{Ker}(\mathcal{L}_f)$ related to the {\em pressure problem} $P(\varphi)$ for a potential $\varphi \in \mathrm{Lip}(\Omega)$. In this case, it is also  natural to consider variations of the form  $f + tg$, with $g \in \mathrm{Lip}(\Omega)$ not necessarily on the Kernel of $\mathcal{L}_f$, and the corresponding directional derivative
\begin{equation} \label{ie123}  \frac{d}{dt}\, \Bigl(\,h( \mu_{f + tg}) + \int_\Omega \varphi\,\, d \mu_{f + tg} \,\Bigr)|_{t=0},\end{equation}
where $f \in \mathcal{N}(\Omega)$ and $\mu_{f + tg}$ is the Lipschitz-Gibbs probability for $f + tg$, with $t$ close to zero.

For fixed $\varphi$, we can ask about the maximal value of  \eqref{ie123} for  $\eta \in \mathrm{Ker}(\mathcal{L}_f)$, when $\eta$ satisfies $\int_\Omega \eta^2 =1$.

\section{Central Limit Theorem}
\label{clt-sec}

Throughout this section, we prove the central limit theorem for our setting. We will show that one can adapt some ideas of the classical case to the $XY$ model. Given the probability space $(\Omega, \mathcal{B}_\Omega, P)$, i.e., $P \in \mathcal{M}_1(\Omega)$, we say that the measurable function $X:\Omega \to \mathbb{R}$ has a {\em Gaussian distribution} with mean $\alpha$ and variance $\sigma^2>0$, when for any interval
 $(a,b)\subset \mathbb{R}$
$$ P(\{x \in \Omega | X(x) \in (a,b)\})= \frac{1}{ \sqrt{2 \pi\, \sigma^2} \,\,}\,\int_a^b \, e^{ - \frac{1 }{2}\, \frac{(t-\alpha)^2 }{\sigma^2}} dt =: \int_a^b\, d \phi_{\alpha,\sigma^2} (t) \;. $$

The former expression is equivalent to saying that
$$ \int_\Omega I_{(a,b)}( X) d P= \int_\Omega I_{X^{-1} (a,b)} d P = \frac{1}{ \sqrt{2 \pi\, \sigma^2} \,\,}\, \int_\mathbb{R} I_{(a,b)}(t) \, e^{ - \frac{1 }{2}\, \frac{(t-\alpha)^2 }{\sigma^2}} dt \;,$$
which implies (via approximation by simple functions) that for any $\varphi \in \mathrm{C}(\mathbb{R})$ the following expression holds
$$ \int_\Omega \varphi(X) d P = \frac{1}{ \sqrt{2 \pi\, \sigma^2} \,\,}\, \int_\mathbb{R} \varphi(t)  \, e^{ - \frac{1 }{2}\, \frac{(t-\alpha)^2 }{\sigma^2}} dt = \int_\mathbb{R} \varphi(t) d \phi_{\alpha,\sigma^2} (t) \;.$$

Then, in this case, the {\em mean} is given by
$$ \alpha := \int_\Omega X \, d P = \frac{1}{ \sqrt{2 \pi\, \sigma^2} \,\,}\, \int_\mathbb{R} t\, \, e^{ - \frac{1 }{2}\, \frac{(t-\alpha)^2 }{\sigma^2}} dt = \int_\mathbb{R} t d \phi_{\alpha,\sigma^2} (t) \;.$$
and  the {\em variance} satisfies the expression
\begin{align*}
\sigma^2 
&:=  \int_\Omega \,(X  - \int_\Omega X\, d P)^2 \, d P \\
&= \frac{1}{ \sqrt{2 \pi\, \sigma^2} \,\,}\, \int_\mathbb{R} (t-\alpha  )^2\, \, e^{ - \frac{1 }{2}\, \frac{(t-\alpha)^2 }{\sigma^2}} dt = \int_\mathbb{R} (t - \alpha)^2 d \phi_{\alpha,\sigma^2} (t) \;.
\end{align*}

\begin{theorem}\label{clt}
Denote $\beta (t) := P(f + tg)- P(f)$, where $f$ and $g$ belong to $\mathrm{Lip}(\Omega)$ and $t \in \mathbb{R}$.  Given $c\in \mathbb{R}, $ the {\em Central Limit Theorem} claims that
\begin{equation} \label{cltla} \lim_{n \to \infty}  \mu_{f} \Bigl(\,\Bigl\{ x \in \Omega \,|\,  \frac{1}{\sqrt{n}} \,
\sum_{j=0}^{n-1} (g(\sigma^j(x)) - \beta^\prime (0)) < c\Bigr\}\,\Bigr) = \int_{-\infty}^c d \phi_{0,\beta^{\prime \prime}(0)} (t) \,.
\end{equation}
The  expression on the first line describes the Gaussian distribution with mean $\alpha=0$ and variance $\sigma^2=\beta^{\prime \prime}\,(0).$
\end{theorem}

Below, we will present a general formulation of the above claim, for the specific proof of the CLT see Theorem \ref{kler}. Before that, we will present some technical results that allow us to prove one of the main theorems of this paper:

\begin{lemma}\label{cltav1}
Denote by $\beta (t) := P(f + tg)- P(f)  = p(t) - p(0)$, where $f, g \in \mathrm{Lip}(\Omega)$, $\int_\Omega g d\mu_f = 0$, and consider $s \in \mathbb{R}$. Then,
\begin{equation} \label{clt1a}
\begin{split}
 \frac{d^2}{ dt^2} p(t)|_{t=s} 
 &= \beta^{\prime \prime}(s) \\
 &= \lim_{n \to \infty} \frac{1}{n}\,\int_\Omega \Bigl(\,\sum_{j=0}^{n-1} \, (\,g \circ \sigma^j - \beta^{\prime}(s) \,) \,\Bigr)^2\, d \mu_{f+ s g} =: \sigma^2_{f + sg} (g).
\end{split}
\end{equation}

The last expression is known as the {\em asymptotic variance} $\sigma^2_{f + sg}(g)$ for the Lipschitz-Gibbs probability $\mu_{f + sg}$.
\end{lemma}

\begin{proof}

In order to simplify the proof we consider the case where $f \in \mathcal{N}(\Omega)$, i.e., when $\mathcal{L}_f(1)=1$. The above implies $w_f \equiv 1$, $P(f) = 0$ and so $\beta (t) = p(t)$. Observe that $\beta^\prime(0) = p^\prime(0) = \int_\Omega g d\mu_f = 0$; we point out that the hypothesis $\int_\Omega g d\mu_f=0$ is   satisfied by tangent vectors to the infinite-dimensional manifold of Lipschitz-Gibbs probabilities, at the point $\mu_f$ (see Theorem \ref{faef7})).

In this case, we have to show that
\begin{equation} \label{clt1r} \beta^{\prime \prime}  (0) = \lim_{n \to \infty} \frac{1}{n}\,\int_\Omega \Bigl(\,\sum_{j=0}^{n-1}  g \,\circ \sigma^j\,\Bigr)^2\, d \mu_{f} \;.
\end{equation}

Indeed, in the same way, such as appears in the last result, we assume
\begin{equation} \label{lulu767}
\mathcal{L}_{f + tg} (w_{f + tg}) = e^{ p(t)} \, w_{f + tg},
\end{equation}
and  $\beta^\prime(0) = p^\prime(0) = \int g d\mu_f = 0$. 

First, note that the derivative at the left side in \eqref{lulu767} satisfies
\begin{align*} \frac{d}{dt} \mathcal{L}_{f + tg} (w_{f + tg})(x)|_{t=0} 
=& \int_M e^{ f (ax)} \,g(ax) d\nu(a) \\
&+ \int_M  e^{ f(ax)} \Bigl(\frac{d}{dt} w_{f + tg} (ax)\Bigr)|_{t=0} d\nu(a) \;.
\end{align*}

In order to estimate $\beta^{\prime \prime}  (0)$, for a fixed $x \in \Omega$, we have to take the derivative
\begin{equation} \label{lulu9}  
\begin{split}
\frac{d^2}{dt^2} \mathcal{L}_{f + tg} (w_{f + tg})(x)|_{t=0} 
=& \frac{d}{dt} \Bigl(\, \int_M  e^{(f + tg) (ax)} \,g(ax) \,w_{f + tg}(ax)d\nu(a) \\
&+ \int_M e^{(f + tg)(ax)}  \frac{d}{dt}  w_{f + tg}(ax)d\nu(a)\Bigr)|_{t=0} \;.
\end{split}
\end{equation}

We consider initially   the first term of \eqref{lulu9}, for a fixed $x \in \Omega$
\begin{align*}
&\frac{d}{dt} \, \Bigl(\int_M  e^{(f + tg) (ax)} \,g(ax) \,w_{f + tg}(ax)d\nu(a)\,\Bigr)|_{t=0} \\
=&\, \int_M  e^{(f + tg) (ax)} \,g(ax)^2 \,w_{f + tg}(ax)d\nu(a)|_{t=0} \\
&+ \, \int_M e^{(f + tg) (ax)} \,g(ax) \,\frac{d}{dt}\,w_{f + tg} (ax)d\nu(a) |_{t=0} \\
=& \, \int_M  e^{f(ax)} \,g(ax)^2 \,d\nu(a) + \int_M e^{f (ax)} \,g(ax) \,\Bigl(\frac{d}{dt}\,w_{f + tg} (ax)\Bigr) |_{t=0} d\nu(a) \\
=& \; \mathcal{L}_ f ( g^2) (x) + \mathcal{L}_ f (g \,  \frac{d}{dt}\,w_{f+tg}|_{t=0})(x) \;.
\end{align*}

Then, integrating w.r.t. $\mu_f$, we have
\begin{equation} \label{lulu181}
\begin{split}
&\int_\Omega \frac{d}{dt} \, \Bigl(\int_M  e^{(f + tg) (ax)} \,g(ax) \,w_{f + tg}(ax)d\nu(a)\,\Bigr)|_{t=0}\, d\mu_f(x) \\
=& \int_\Omega \mathcal{L}_ f ( g^2) d\mu_f\,+ \int_\Omega \mathcal{L}_ f (g \,  \frac{d}{dt}\,w_{f+tg}|_{t=0}) \,d\mu_f \\
=& \int_\Omega g^2 d\mu_f + \int_\Omega g \,.\,\Bigl( \frac{d}{dt}\,w_{f + tg}\Bigr)|_{t=0} \, d\mu_f \;.
\end{split}
\end{equation}

Now we estimate the second term of \eqref{lulu9}, also for a fixed $x \in \Omega$
\begin{align*}
&\frac{d}{dt} \, \Bigl(\int_M e^{(f + tg)(ax)}  \frac{d}{dt}  w_{f + tg}(ax)d\nu(a)\Bigr)|_{t=0} \\
=& \, \int_M  e^{(f + tg) (ax)} \,g(ax) \,\frac{d}{dt}  w_{f + tg} (ax) d\nu(a)|_{t=0} \\
&+ \int_M e^{(f + tg) (ax)}  \, \frac{d^2}{dt^2} w_{f + tg} (ax) d\nu(a)|_{t=0} \\
=& \; \mathcal{L}_{ f}\, (\,g \, \frac{d}{dt}  w_{f + tg} |_{t=0}\,)(x) + \mathcal{L}_{ f}\, (\frac{d^2}{dt^2}  w_{f + tg}|_{t=0}) (x) \;.
\end{align*}

Therefore, integrating w.r.t. the measure $\mu_f$, it follows that
\begin{equation} \label{lulu190} 
\begin{split}
& \int_\Omega \,  \frac{d}{dt} \, \Bigl(\int_M e^{(f + tg)(ax)}  \frac{d}{dt}  w_{f + tg}(ax)d\nu(a)\Bigr)|_{t=0} \, d\mu_f (x) \\
=& \int_\Omega \mathcal{L}_{ f}\, (\,g \, \frac{d}{dt}  w_{f + tg} |_{t=0}\,) \, d\mu_f + \int_\Omega \mathcal{L}_{ f}\, (\frac{d^2}{dt^2}  w_{f + tg}|_{t=0}) \,d\mu_f \\
=& \int_\Omega \,g \,.\, \Bigl(\frac{d}{dt}  w_{f + tg}\Bigr) |_{t=0}\, d\mu_f + \int_\Omega \frac{d^2}{dt^2}  w_{f + tg}|_{t=0} \,d\mu_f \;.
\end{split}
\end{equation}

Therefore, from \eqref{lulu181} and \eqref{lulu190} we get
\begin{equation} \label{lulu943}
\begin{split}
\int_\Omega \frac{d^2}{dt^2} \mathcal{L}_{f + tg} (w_{f + tg})|_{t=0} d\mu_f 
=& \int_\Omega g^2 d\mu_f + 2\, \int_\Omega \,g \,.\, \Bigl(\frac{d}{dt}  w_{f + tg}\Bigr) |_{t=0}\, d\mu_f \\
&+ \int_\Omega \frac{d^2}{dt^2}  w_{f + tg}|_{t=0} \,d\mu_f \;.
\end{split}
\end{equation}

On the other hand, since we are assuming $\beta^\prime(0) = p^\prime(0) = \int g d\mu_f = 0$, for each $x$ the following holds
\begin{align*}
& \frac{d^2}{dt^2} (e^{ p(t)} \, w_{f +tg} (x)) |_{t=0} =  \frac{d}{dt} \,\Bigl( \frac{d}{dt} (e^{ p(t)}) \, w_{f + tg} (x)+ e^{ p(t)}\frac{d}{dt}  w_{f + tg}(x)\Bigr)|_{t=0} \\
=& \,\frac{d^2}{dt^2} e^{ p(t)}|_{t=0} w_f(x) + 2\frac{d}{dt} e^{ p(t)}|_{t=0} \frac{d}{dt} w_{f + tg} (x) |_{t=0} + e^{ p(0)}\,\frac{d^2}{dt^2}  w_{f + tg}(x)|_{t=0} \\
=& \,\frac{d^2}{dt^2} e^{ p(t)}|_{t=0} + 2\int_\Omega g d\mu_f \frac{d}{dt} w_{f + tg} (x) |_{t=0} + \frac{d^2}{dt^2}  w_{f + tg}(x)|_{t=0} \\ 
=&  \,\frac{d^2}{dt^2} e^{ p(t)}|_{t=0} + \frac{d^2}{dt^2}  w_{f + tg}(x)|_{t=0} \;.
\end{align*}

Then,
\begin{equation} \label{lulu400}
\begin{split}
&\int_\Omega \frac{d^2}{dt^2} (e^{ p(t)} \, w_{f +tg}) |_{t=0} \, d\mu_f \\
=& \, \int_\Omega   \,\frac{d^2}{dt^2} e^{ p(t)}|_{t=0}\, d\mu_f   +\,\int_\Omega \,\frac{d^2}{dt^2}  w_{f + tg}|_{t=0} \, d\mu_f \\
=& \,\frac{d^2}{dt^2} e^{ p(t)}|_{t=0}    + \,\int_\Omega \,\frac{d^2}{dt^2}  w_{f + tg}|_{t=0} \, d\mu_f \;.
\end{split}
\end{equation}

Besides that, since $f \in \mathcal{N}(\Omega)$, we have
$$\frac{d^2}{dt^2}e^{p(t)}|_{t=0} = p^{\prime \prime}(0) + (p^{\prime}(0))^2 =  p^{\prime \prime}(0) \;.$$

Therefore, from  \eqref{lulu943}, \eqref{lulu400}, and the assumption $\beta^\prime(0) = p^\prime(0) = \int_\Omega g d\mu_f = 0$, we get that
\begin{equation} \label{lulu401}  p^{\prime\prime}(0) = \,\frac{d^2}{dt^2} e^{ p(t)}|_{t=0} =  \int_\Omega g^2 d\mu_f + 2\, \int_\Omega \,g \,.\, \Bigl(\frac{d}{dt}  w_{f + tg}\Bigr) |_{t=0}\, d\mu_f \;.
\end{equation}

Furthermore, given $n \in \mathbb{N}$, following the same reasoning as above, but now for the expression
$$ \mathcal{L}_{ f + tg}^n (w_{f + tg}) = e^{ n\, p(t)} \, w_{t + fg} \;,$$
we get that
\begin{equation} \label{lulu402} 
\begin{split}
&n\,p^{\prime \prime}(0) = \,n\,\frac{d^2}{dt^2} e^{ p(t)}|_{t=0} \\
=& \int_\Omega  \Bigl(\sum_{j=0}^{n-1}\,g \circ  \sigma^j\Bigr)^2 d\mu_f + 2\, \int_\Omega  \Bigl(\,\sum_{j=0}^{n-1} g \circ  \sigma^j\Bigr) \,.\, \Bigl( \frac{d}{dt}\,w_{f + tg}\Bigr)|_{t=0} \, d\mu_f \;.
\end{split}
\end{equation}

From \eqref{lulu402} it follows that for all $n \in \mathbb{N}$ the following holds
\begin{equation} \label{lulu403}
\begin{split}
&p^{\prime \prime}(0) = \,\frac{d^2}{dt^2} e^{ p(t)}|_{t=0} \\ 
=& \frac{1}{n}\int_\Omega \Bigl(\sum_{j=0}^{n-1}\,g \circ  \sigma^j\Bigr)^2 d\mu_f + \frac{2}{n}\, \int_\Omega \Bigl(\,\sum_{j=0}^{n-1}g \circ  \sigma^j\Bigr) \,.\, \Bigl( \frac{d}{dt}\,w_{f + tg}\Bigr)|_{t=0} \, d\mu_f \;.
\end{split}
\end{equation}

As $\int_\Omega g d\mu_f =0$ and $\frac{d}{dt}\,w_{f + tg}(\cdot)|_{t=0}$ is a bounded function on $\Omega$, it follows from the Ergodic Theorem (and Dominated Convergence Theorem) that
\begin{equation} \label{lulu404} \beta^{\prime \prime} (0) = p^{\prime \prime} (0) = \frac{d^2}{dt^2} e^{ p(t)}|_{t=0} =
\lim_{n \to \infty} \frac{1}{n}  \int_\Omega  \Bigl(\sum_{j=0}^{n-1}\,g \circ  \sigma^j\Bigr)^2 d\mu_f \;.
\end{equation}

Now, taking $r := t-s$, $\psi := f + sg$ and calculating $\frac{d^2}{dr^2} e^{ p(r)}|_{r=0}$, we obtain that
$$
\beta^{\prime \prime}(s) = p^{\prime \prime} (s) = \lim_{n \to \infty} \frac{1}{n}\,\int_\Omega \Bigl(\,\sum_{j=0}^{n-1} \, (\,g \circ \sigma^j - \beta^{\prime}(s)\,) \,\Bigr)^2\, d \mu_{f+ s g}
$$
 
\end{proof}

The former Theorem implies the following:

\begin{corollary}\label{cltav2} Consider $\beta (t) := P(f + tg)- P(f)$, $t \in \mathbb{R}$, where $f, g \in \mathrm{Lip}(\Omega)$ and $\int_\Omega g d\mu_f = 0$. Then,
\begin{equation} \label{lulu4015}
\begin{split}
\sigma^2_f (g) &:=  \beta^{\prime \prime} (0)   = \lim_{n \to \infty} \frac{1}{n}\,\int_\Omega \,\Bigl(\sum_{j=0}^{n-1}  g \,\circ \sigma^j \Bigr)^2\, d \mu_{f} \\ 
=& \int_\Omega g^2 d\mu_f + 2\, \int_\Omega g \,.\, \Bigl( \frac{d}{dt}\,w_{f + tg}\Bigr)|_{t=0} \, d\mu_f \;.
\end{split}
\end{equation}
\end{corollary}

The value $ \beta^{\prime \prime}  (s)$ stated at Lemma \ref{cltav1} is called the {\em asymptotic variance} for the Lipschitz-Gibbs probability $\mu_{f + sg}$. This value is the one to be used as a variance of the Gaussian distribution obtained  in the central limit theorem for the Lipschitz-Gibbs probability $\mu_{f}$ (see \cite{PP}, \cite{CoPa} or \cite{Chu}). Our proof adapts  the reasoning used in  \cite{Lal1} for our setting.

\begin{remark}\label{falfa} Note that in the case when $g$ is co-boundary to a constant  $c \in \mathbb{R}$, i.e., there exists $v \in \mathrm{C}(\Omega)$ such that $g = v - v \circ \sigma + c$. Then,  for any $s \in \mathbb{R}$, we have $p^\prime(s) = c$ and, therefore,  for any $s \in \mathbb{R}$, it follows that $p^{\prime \prime}(s) = 0$. In Proposition 4.2 in \cite{PP} it was shown the converse, i.e., when, $p^{\prime \prime}(s) = 0$ for all $s \in \mathbb{R}$, then, $g$ is co-boundary to a constant.
\end{remark}

We will assume from now on that $g \in \mathrm{Lip}(\Omega)$ and not co-boundary to a constant. In this case, it follows from Remark \ref{falfa} that the map $t \mapsto P( f + t g)$ is strictly convex. 

Before presenting the main theorem of this paper, we will introduce the following well-known technical result which allows us to prove properties of convergence of a family of probability distributions using the point-wise convergence of their corresponding moment generation functions (for more details see Theorem 2 in section XIII in \cite{Feller}).

\begin{theorem} {(Continuity Thm \cite{Feller,Miller})} \label{cthm}
Consider a sequence of Borel probability densities $(\rho_n)_{n = 1}^\infty$ on $\mathbb{R}$ and its corresponding {\em moment generation functions}  given by $\varphi_n (z) := \int_{\mathbb{R}} e^{\,z \,t} d \rho_n (t)$. If for each value $z \in \mathbb{R}$  we have 
$$\lim_{n \to \infty} \varphi_n(z) = \varphi(z) :=\, \int_{\mathbb{R}} e^{\,z \,t} d \phi_{\alpha,\sigma^2} (t) = e^{z\alpha + z^2\frac{ \sigma^2}{2}}\;.$$ 
That is, $\varphi$ is the moment generation function of the Gaussian density $\phi_{\alpha,\sigma^2}$ on $\mathbb{R}$, then, the sequence of probability densities $(\rho_n)_{n = 1}^\infty$ converges weakly to the Gaussian probability $\phi_{\alpha,\sigma^2}$ when $n\to \infty$.
\end{theorem}

\begin{remark}\label{rclt}
The former theorem implies that for, either a simple function of the form $u := \sum_{j=1}^k a_j 1_{A_j}$ or a function $u \in \mathrm{C}(\mathbb{R})$, we have
$$\lim_{n \to \infty} \int_\Omega u(t) d\rho_n(t) = \int_\Omega u(t) d\phi_{\alpha,\sigma^2}(t) \;.$$
\end{remark}

Now we are able to state the main result of this section, which is a very generalized version of the CLT.

\begin{theorem} \label{kler}
Denote by $\phi_{\alpha,\sigma^2}$ the Gaussian density on $\mathbb{R}$ with mean $\alpha$ and variance $\sigma^2$, consider $\beta (t) := P(f + tg)- P(f)$, so, $\beta^\prime(0)$ and $\sigma^2 = \beta^{\prime \prime}(0) $. Let $u$ be, either a simple function $u := \sum_{j=1}^k a_j 1_{A_j}$  or a function $u \in \mathrm{C}(\mathbb{R})$. Then, we have
 \begin{equation} \label{clt2}
 \lim_{n \to \infty} \int_\Omega u\Bigl(\,\, \frac{1}{\sqrt{n}}\,\sum_{j=0}^{n-1}  (\,g \circ \sigma^j - \beta^{\prime}(0)\,) \,\Bigr)\,\, d \mu_f = \int_{\mathbb{R}} u(t)\,\, d \phi_{0,\beta^{\prime \prime}(0)} (t) \;.
 \end{equation}
\end{theorem}

\begin{proof}
It is well-known that the moment generation function for the Gaussian distribution $\phi_{0,\beta^{\prime \prime}(0)}$ with mean $0$ and variance $\beta^{\prime \prime}(0)$ is given by the expression
$$\varphi(z) :=\, \int_{\mathbb{R}} e^{\,z \,t} d \phi_{0,\beta^{\prime \prime}(0)} (t) = e^{z^2\frac{\beta^{\prime \prime}(0)}{2}} \;.$$

For each $n \in \mathbb{N}$, denote by $\rho_n$ the Borel probability measure on $\mathbb{R}$, such that, the interval $(a,b)$ has $\rho_n$-measure equal to
$$ \rho_n (a, b) := \mu_f \Bigl( \Bigl\{\, x \in \Omega\,|\, \frac{1}{\sqrt{n}}\,\sum_{j=0}^{n-1}  (\,g(\sigma^j(x)) - \beta^\prime (0)\,) \,\in (a,b)\, \Bigr\}\,\Bigr) \,.$$

Let $\varphi_n$ be the corresponding moment generation function associated with the probability distribution $\rho_n$, i.e., the one satisfying for each $z \in \mathbb{R}$
$$\varphi_n(z) = \int_{\mathbb{R}} e^{\,z \,t} d \rho_n(t) \;.$$

We will show that for $z$ fixed, we get that 
$$\lim_{n \to \infty} \varphi_n(z) = \varphi(z) := \, e^{z^2 \, \frac{\beta^{\prime \prime}(0)}{2}}\;,$$ that is, the map $\varphi$ is the moment generation function of $\phi_{0,\beta^{\prime \prime}(0)}$.

At the first, consider the Taylor's polynomial of order $2$ at zero for the map $\beta(t)$, which is given by
 \begin{equation} \label{Tay2} \beta(t) := \beta^{\prime}(0) \, t +  \frac{1}{2} \beta^{\prime \prime}(0) \, t^2 + \, o(t^3)\;,
 \end{equation}
 where $\beta^\prime (0)=\int_\Omega g d \mu_f $   and $\beta^{\prime \prime}(0) $ is given by \eqref{lulu4015}. 
 
 Besides that, assume that $\int_\Omega w_{f + tg} d \rho_{f +tg}=1$, for all $t \in \mathbb{R}$, where $w_{f + tg}$  and $\rho_{f + tg}$ satisfy the expressions appearing in \eqref{rpf1}. 
 
Observe that $\rho_n = (\gamma_n)_* \mu_f$, with $\gamma_n := \frac{1}{\sqrt{n}}\,\sum_{j=0}^{n-1}  (\,g \circ \sigma^j - \beta^\prime (0)\,)$. Then, the moment generating function for the Borel probability measure $\rho_n$ is given by the expression

\begin{equation} \label{super1}
\begin{split}
\varphi_n (z) 
&= \int_\mathbb{R} e^{\,z t} d \rho_n (t) \\
&= \int_\Omega e^{\,z\, \frac{1}{\sqrt{n}}\,\sum_{j=0}^{n-1}  (\,g \circ \sigma^j - \beta^\prime (0)\,)} d \mu_f \\ 
&= e^{-\,z\,\frac{n}{\sqrt{n}}\beta^\prime (0)}\,\,\int_\Omega \,e^{z\,\frac{1}{\sqrt{n}}\sum_{j=0}^{n-1} \,g \circ \sigma^j}\, w_f \, \frac{1}{ \lambda_f^n} d(\,(\mathcal{L}_f^*)^n  (\rho_f)) \\ 
&= e^{-\,z\,\frac{n}{\sqrt{n}}\beta^\prime (0)}\,\frac{1}{ \lambda_f^n}\int_\Omega \mathcal{L}_{f+ z\,\frac{1}{\sqrt{n}}\, g}^n\,(\, w_f)\, d\,   \rho_f \\
&= e^{-\,z\,\frac{n}{\sqrt{n}}\,\beta^\prime (0)}\frac{1}{ \lambda_f^n}\,\frac{\lambda^n_{f+ z\,\frac{1}{\sqrt{n}}\, g}}{ \lambda^n_{f+ z\,\frac{1}{\sqrt{n}}\, g}} \int_\Omega \mathcal{L}_{f+ z\,\frac{1}{\sqrt{n}}\, g}^n\,(\, w_f) \, d\,  \rho_f \\
&= e^{-\,z\,\frac{n}{\sqrt{n}}\,\beta^\prime (0)\,+\,n\, \beta(  z\,\frac{1}{\sqrt{n}} )  }\,\frac{1}{ \lambda^n_{f+ z\,\frac{1}{\sqrt{n}}\,g}} \int_\Omega \mathcal{L}_{f+ z\,\frac{1}{\sqrt{n}}\,g}^n\,(\, w_f)\, d\,   \rho_f  \\
&= e^{n (\, -\,z\,\beta^\prime (0)\frac{1}{\sqrt{n}}\,+\, \beta(  z\,\frac{1}{\sqrt{n}} )  \,)}\int_\Omega \frac{\mathcal{L}^n_{f+ z\,\frac{1}{\sqrt{n}}\,g}\,(\, w_f) }{  \lambda^n_{f+ z\,\frac{1}{\sqrt{n}}\,g}}\, d\,   \rho_f \;. 
\end{split}
\end{equation}

On the other hand, by the continuity of the maps $f \mapsto w_f$ and $f \mapsto \rho_f$, it follows immediately that
 \begin{equation} \label{super133} \lim_{n \to \infty}w_{f+ z\,\frac{1}{\sqrt{n}}\,g} = w_f \,\,\text{ and}\,\,  \lim_{n \to \infty}\rho_{f+ z\,\frac{1}{\sqrt{n}}\,g} = \rho_f \,.
 \end{equation}

 In order to estimate \eqref{super1}, observe that taking the potential $f + z\,\frac{1}{\sqrt{n}}\,g$ instead of $f$ in \eqref{rpf2} and choosing $w_f$ instead of $w$ in \eqref{rpf2} as well, by the assumption $\int_\Omega w_f \,d\, \rho_{f} =1$, it follows that any $x \in \Omega$ satisfies the limit
\begin{equation} \label{super134}
\begin{split}
\lim_{n \to \infty}   \frac{\mathcal{L}^n_{f+ z\,\frac{1}{\sqrt{n}}\,g}\,(\, w_f)  \,(x)}{  \lambda^n_{f+ z\,\frac{1}{\sqrt{n}}\,g}}
&= \lim_{n \to \infty}  w_{f+ z\,\frac{1}{\sqrt{n}}\,g} (x) \, \int_\Omega w_f \,d\, \rho_{f+ z\,\frac{1}{\sqrt{n}}\,g} \\
&= w_f(x) \,  \int_\Omega w_f \,d\, \rho_{f}= w_{f} \,(x)\, .
\end{split}
\end{equation}

Besides that, since $\|\mathcal{L}^n_{f+ z\,\frac{1}{\sqrt{n}}\,g}\,(\, w_f)\|_\infty < \infty$ for any $n \in \mathbb{N}$, we can take into account the reasoning in \eqref{super134} for estimating the limit in  \eqref{super1}. 

Indeed, first note that $\lim_{n\to \infty} \beta(  z\,\frac{1}{\sqrt{n}}) = \beta(0)= 0$. Therefore, up to a multiplicative constant, from \eqref{Tay2}, \eqref{super1}, \eqref{super133}, \eqref{super134} and the assumption $\int w_f \,d\, \rho_{f} =1$, it will follow that
 \begin{equation} \label{super2}
 \begin{split}
 \lim_{n \to \infty} \varphi_n (z) 
 &= \lim_{n\to\infty} \int_\Omega e^{-\,z t} d \rho_n (t) \\
 &= \lim_{n\to\infty} e^{n (\,-\,z\,\beta^\prime (0)\frac{1}{\sqrt{n}}\,+\, \beta(  z\,\frac{1}{\sqrt{n}} )  \,)}\int_\Omega \frac{\mathcal{L}^n_{f+ z\,\frac{1}{\sqrt{n}}\,g}\,(\, w_f)\,}{  \lambda^n_{f+ z\,\frac{1}{\sqrt{n}}\,g}}\, d\,\rho_f \, \\
 &= e^{ z^2\, \frac{1}{2}\,\beta^{\prime \prime} (0)} \, \lim_{n\to\infty}  \int_\Omega \frac{\mathcal{L}^n_{f+ z\,\frac{1}{\sqrt{n}}\,g}\,(\, w_f)\,}{  \lambda^n_{f+ z\,\frac{1}{\sqrt{n}}\,g}}\, d\,   \rho_f\, \\
 &= e^{ z^2\, \frac{1}{2}\,\beta^{\prime \prime} (0)} \, \,\int_\Omega w_f d\,\rho_f = e^{ z^2\, \frac{1}{2}\,\beta^{\prime \prime} (0)} \,,
  \end{split}
 \end{equation}
 where $e^{ z^2\, \frac{1}{2}\,\beta^{\prime \prime} (0)}$ is the moment generation function for the Gaussian probability distribution  $\phi_{0, \beta^{\prime \prime} (0)}.$

 Finally, from the above, we get that the following limit holds
 $$ \lim_{n \to \infty}\varphi_n (z) = e^{ z^2\, \frac{1}{2}\,\beta^{\prime \prime} (0)} = \, \int_{\mathbb{R}} e^{\,z \,t} d \phi_{0,\beta^{\prime \prime}(0)} (t)\;.$$
 
f So, our claim holds by Theorem \ref{cthm}.
\end{proof}
 
 \begin{remark}
Theorem \ref{kler} implies the following interesting cases: If we take $u(x) := 1_{(-\infty, c)}$ we get that \eqref{clt2} is the same as the expression in \eqref{cltla}; when we take $u(x) := x$ we obtain that the limit in \eqref{clt2} is equal to $\beta^\prime(0) = \int_\Omega g d \mu_f$; moreover, choosing $u(x)=x^2$, it follows that the limit in  \eqref{clt2} is equal to the value $\beta^{\prime \prime}(0) = \lim_{n \to \infty} \frac{1}{n}  \int_\Omega  (\sum_{j=0}^{n-1}\,g \circ  \sigma^j)^2 d\mu_f$.
\end{remark}

\section{First and Second Derivatives} 
\label{Seco}

Given an {\em a priori measure} $\nu \in \mathcal{M}_1(M)$ and  a  potential $f \in \mathrm{Lip}(\Omega)$, consider as in \eqref{ruop}  the Ruelle operator $\mathcal{L}_f$ given by

\begin{equation*}
\mathcal{L}_f (w)(x) := \int_M e^{f(ax)}w(ax) d\nu(a).
\end{equation*}

Assume that $f \in \mathcal{N}(\Omega)$ and denote by $\mu_f$ the probability such that $\mathcal{L}_f^*( \mu_f)= \mu_f$.  Remember that the set of all $\mu_f$'s, obtained when $f \in \mathcal{N}(\Omega)$,  is also denoted by  $\mathcal{N}(\Omega)$ and any element in this set is called a Lipschitz-Gibbs probability.

The entropy of $\mu_f \in \mathcal{N}(\Omega)$, where $f = \log J$ (see \cite{LMMS}), is given by the expression
\begin{equation}
\label{ruop7777} h(\mu_f) = - \int_\Omega \log J d \mu_f=  - \int_\Omega f d \mu_f \;.
\end{equation}

 Denote by $N : \mathrm{Lip}(\Omega) \to \mathcal{N}(\Omega)$ the function sending a $f \in \mathrm{Lip}(\Omega)$  into its associated normalized potential $N(f) =: \log J_f \in \mathcal{N}(\Omega)$, it is well known that the function $f \mapsto N(f)$ is analytic (see Theorem 3.5 in \cite{GKLM}). Denote by $\mathcal{C}$ the set of Lipschitz continuous functions $\psi : \Omega \to \mathbb{R}$ of the form $\psi = w - (w \circ \sigma) + c$, where $c \in \mathbb{R}$ and $w \in \mathrm{Lip}(\Omega)$.

Theorem 3.3 in \cite{GKLM} claims that for $f \in \mathcal{N}(\Omega)$, the vector spaces $\mathrm{Ker}(\mathcal{L}_f)$ and $\mathcal{C}$ intersect trivially and, even more
$$\mathrm{Lip}(\Omega) = \mathrm{Ker}(\mathcal{L}_f) \oplus \mathcal{C} \;.$$

Therefore, any  potential $\varphi \in \mathrm{Lip}(\Omega)$, can be written in the form
\begin{equation} \label{kluy}
\varphi = \xi + \Bigl(w - (w \circ \sigma) + c \Bigr),
\end{equation}
with $\xi \in \mathrm{Ker}(\mathcal{L}_f)$ and $w := (I - \mathcal{L}_f)^{-1}(c - \mathcal{L}_f(\varphi)) \in \mathrm{Lip}(\Omega)$.

Moreover, if  $\varphi \in \mathcal{N}(\Omega)$ is co-boundary to $\psi \in \mathcal{N}(\Omega)$, then, the associated $\xi$ appearing in \eqref{kluy} is the same for both $\varphi$ and $\psi$.

For a function $\eta \in \mathrm{Ker}(\mathcal{L}_f)$ which satisfies $\int_\Omega \eta d \mu_f = 0$, a natural normalization is to assume that $\int_\Omega \eta^2 d \mu_f = 1$.

\medskip

\begin{remark} \label{ptre}
 Given $f, g \in \mathrm{Lip}(\Omega)$ and  a function  $\varphi \in \mathrm{Lip}(\Omega)$ in the form  \eqref{kluy} note that
$$ \int_\Omega \varphi \, d \mu_{f + g} - \int_\Omega \varphi \, d \mu_f = \int_\Omega \xi \, d \mu_{f + g} - \int_\Omega \xi \, d \mu_f \;.$$

Therefore, for computations like the calculus of the derivatives
$$ \frac{d}{dt} \int_\Omega \varphi \, d \mu_{f + tg} $$
we can assume that $\varphi \in \mathrm{Ker}(\mathcal{L}_f)$, i.e., to take $\varphi = \xi$.
\end{remark}

\medskip

\begin{remark} \label{ptre}  Note that when $\rho \in \mathrm{Ker}(\mathcal{L}_f)$, we get that
\begin{equation} \label{vaga5} (I - \mathcal{L}_f)^{-1} \rho = \rho.
\end{equation}
\end{remark}

\begin{theorem}{[Thm 3.5 in \cite{GKLM}]} The derivative of the map $N$ at the point $f$, denoted by $DN_f$, is the linear projection of the space $\mathrm{Lip}(\Omega)$ on $\mathrm{Ker}(\mathcal{L}_{N(f)})$  in the direction of $\mathcal{C}$.

\end{theorem}

\begin{theorem}{\cite{GKLM}} \label{faef1}
Consider the map $\eta \mapsto \mu_{f + \eta}$, where $f, \eta \in \mathrm{Lip}(\Omega)$. Assume that $f \in \mathcal{N}(\Omega)$, $\varphi \in \mathrm{Lip}(\Omega)$ satisfies $\int_\Omega \varphi d \mu_f = 0$, and  finally that $\eta \in \mathrm{Ker}(\mathcal{L}_f)$. Then,
\begin{equation} \label{lind1}
\int_\Omega \varphi \,d \mu_{f + \eta}= \int_\Omega (I - \mathcal{L}_f)^{-1} (\varphi) \eta d \mu_f + O(||\varphi|| \, ||\eta||^2) \,.
\end{equation}

Now, assuming that $f, g \in \mathrm{Lip}(\Omega)$ are general (not necessarily in the kernel of $\mathcal{L}_f$),
we get from Theorem C (or expression (6)) in \cite{GKLM}

\begin{equation} \label{lind2}
\frac{d}{d t} \int_\Omega \varphi \, d \mu_{f + tg}|_{t=0}= \int_\Omega (I - \mathcal{L}_f)^{ -1} (\varphi_f) \,. \, DN_f (g) d \mu_f
\end{equation}
where $\varphi_f =\varphi - \int_\Omega \varphi d \mu_f.$

Finally, if we consider $f \in \mathcal{N}(\Omega)$, $\eta \in \mathrm{Ker}(\mathcal{L}_f)$, and do not assume that $\int_\Omega \varphi d\mu_f = 0$, we get the more general result
\begin{equation} \label{lind4}
\int_\Omega \varphi \, d \mu_{f + \eta} = \sum_{j=0}^\infty  \int_\Omega \varphi \,\, (\eta \circ \sigma^j) \, d \mu_f + O(||\varphi|| \,||\eta||^2) \,.
\end{equation}
\end{theorem} 

\begin{proof} The first claim  \eqref{lind1} and the last claim \eqref{lind4} follow from Theorem 4.2 in \cite{GKLM}. The claim in \eqref{lind2} follows from Theorem 4.1 in \cite{GKLM}.
\end{proof}

A useful result follows from the above:

\begin{theorem}{\cite{GKLM}} \label{faef2}  Assume that $f, \eta, \varphi \in \mathrm{Lip}(\Omega)$, $f \in \mathcal{N}(\Omega)$, and consider the law $t \mapsto \mu_{f + t\eta}$, where $\eta \in \mathrm{Ker}(\mathcal{L}_f)$. Then,
\begin{equation} \label{lind5}
\frac{d}{dt} \int_\Omega \varphi \, d \mu_{f + t\eta}|_{t=0} = \int_\Omega \varphi \,\, \eta\, d \mu_f \,.
\end{equation}
\end{theorem}
\begin{proof} For a proof see (2) in Theorem 5.1 in \cite{GKLM}.

\end{proof}

\begin{corollary} \label{vaga2}  Assume $f, g, \varphi \in \mathrm{Lip}(\Omega)$, with $f \in \mathcal{N}(\Omega)$. Consider the law $t \mapsto \mu_{f + tg}$, where $g$ does not necessarily belong to the kernel of $\mathcal{L}_f$.  Assume that there are $v, w \in \mathrm{C}(\Omega)$ such that $g = \eta + (v - v \circ \sigma)$ and $\varphi= \xi + (w - w \circ \sigma) +c$, where $\eta, \xi \in \mathrm{Ker}(\mathcal{L}_f)$, and $c \in \mathbb{R}$. Then , we have,
\begin{equation} \label{lind5}
\frac{d}{dt} \int_\Omega \varphi \, d \mu_{f + tg}|_{t=0}= \int_\Omega \,\, \xi\, \eta\, d \mu_f \;.
\end{equation}
\end{corollary}

\begin{proof} From \eqref{kluy} and Remark \ref{ptre} we can take 
$$ \frac{d}{dt} \int_\Omega \varphi \, d \mu_{f + tg}|_{t=0} = \frac{d}{dt} \int_\Omega \xi \, d \mu_{f + tg}|_{t=0}.$$

From \eqref{lind2}
\begin{equation} \label{lind231}
\begin{split}
\frac{d}{d t} \int_\Omega \xi \, d \mu_{f + tg}|_{t=0} 
&= \int_\Omega (I - \mathcal{L}_f)^{-1} (\xi) \,.\, DN_f (g) d \mu_f \\
&= \int_\Omega \xi\,.\, \eta\, d \mu_f \;.
\end{split}
\end{equation}

\end{proof}

A different proof of \eqref{ddent1} was presented in \cite{GKLM}.

\index{derivative of entropy}

In the following proposition, we study the behavior of the entropy  $h(\mu_f)$ for a Lipschitz-Gibbs probability $\mu_f \in \mathcal{N}(\Omega)$.

\begin{proposition} \label{dent} Consider the path $t \mapsto \mu_{f + t\eta}$, where $f$ and $\eta$ belong to $\mathrm{Lip}(\Omega)$. Assume that $f =\log J \in \mathcal{N}(\Omega)$ and $\eta \in \mathrm{Ker}(\mathcal{L}_f)$. Denoting $h(t) :=h (\mu_{f + t\eta})$ the entropy of the Lipschitz-Gibbs probability $\mu_{f + t\eta}$, we get
\begin{equation} \label{ddent1}  \frac{d}{d t} h(t)|_{t=0} \,=\, - \int_\Omega f \, \eta \,d \mu_f \;,
\end{equation}
and
\begin{equation} \label{ddent2} \frac{d^2}{dt^2} h(t)|_{t=0}  \,=  - \Bigl(\, \int_\Omega \, \eta^2\, d \mu_f + \int_\Omega f \, \eta^2\, d \mu_f\Bigr) =   - \, \int_\Omega (\log J + 1) \, \eta^2\, d \mu_f \;.
\end{equation}
\end{proposition}

\begin{proof} First note that it follows from the chain rule and  Theorem \ref{faef2}:
\begin{equation} \label{ddent3}
\begin{split}
 \int_\Omega \eta\, d \mu_{f + t\eta} 
&= \frac{d}{dt} P(f + t\eta) =\frac{d}{dt} \Bigl( h(t) + \int_\Omega (f + t\eta)\, d \mu_{f + t\eta} \Bigr) \\
=& \frac{d}{dt}  h(t) \,+\, \int_\Omega \eta\,  d \mu_{f + t\eta}\, +\, \int_\Omega (f + t\eta)\,\eta\,  d \mu_{f + t\eta} \;.
\end{split}
\end{equation}

Since $\eta \in \mathrm{Ker}(\mathcal{L}_f)$, we get $\int_\Omega \eta d \mu_f = 0$. This implies that
$$\frac{d}{dt}  h(t) |_{t=0}\,=\,  -\, \int_\Omega f \,\eta d \mu_f \;.$$

Furthermore, from \eqref{ddent3}, it follows that
\begin{align*}
\frac{d^2}{dt^2} h(t)|_{t=0} 
=&\, - \, \frac{d}{dt} \, \Bigl( \, \int_\Omega (f + t\eta)\,\eta\,  d \mu_{f + t\eta}\, \Bigr)|_{t=0} \\
=&\, - \,\Bigl(\,\int_\Omega f \, \eta^2\, d \mu_f + \int_\Omega \, \eta^2\, d \mu_f \,\Bigr) \\
=&\, -\int_\Omega (\log J + 1) \, \eta^2\, d \mu_f \\
=& - \,\,\int_\Omega (\log J + \log (e)) \, \eta^2\, d \mu_f 
=\, - \int_\Omega \log (J \,e)  \, \eta^2\, d \mu_f \;.
\end{align*}

 It is natural to assume a normalization condition $\int_\Omega \eta^2\, d \mu_f = 1.$ In this case
$$-\int_\Omega \log (J \,e)  \  \, \eta^2\, d \mu_f \geq - \int_\Omega \log \,(e)  \  \, \eta^2\, d \mu_f = - 1.$$

The function $\eta$  may enhance points $x$ where the value $(J(x) \,e) $ is close to zero (producing large values for $\frac{d^2}{dt^2} h(t)|_{t=0})$) or points  $x$ where the value $(J(x) \,e) $ is close to $e$ (producing  values $\frac{d^2}{dt^2} h(t)|_{t=0})$ close to $-1$).
\end{proof}

\bigskip

\begin{proposition}  \label{vaga1}  Assume, $f, g \in \mathrm{Lip}(\Omega)$, with $g$ not necessarily on $\mathrm{Ker}(\mathcal{L}_f)$. Then,
\begin{equation} \label{lent1}  \frac{d}{dt}h(\mu_{f + tg})=-\,   \int_\Omega \zeta \,\eta\,  d \mu_f \;, \end{equation}
where $f = \zeta + (u - u \circ \sigma) + c$ and $g = \eta + (v - v \circ \sigma)$, with $\zeta, \eta \in \mathrm{Ker}(\mathcal{L}_f)$ and $c \in \mathbb{R}$.
\end{proposition}

\begin{proof}
 If $g$ is not necessarily on the kernel of $\mathcal{L}_f$, we get from \eqref{lind2} that
\begin{equation} \label{iddent3} 
\begin{split}
\int_\Omega g\, d \mu_f 
=& \frac{d}{dt} P(f + tg)|_{t=0} \\
=& \frac{d}{dt} \Bigl( h(\mu_{f + tg}) + \int_\Omega (f + tg)\, d \mu_{f + tg} \Bigr)|_{t=0} \\
=& \frac{d}{dt}  h(\mu_{f + tg})|_{t=0} \,+\, \int_\Omega g d\mu_f \\
&+ \int_\Omega (I - \mathcal{L}_f)^{-1}\Bigl(f - \int_\Omega f d\mu_f\Bigr)\,.\,DN_f(g) d\mu_f \;.
\end{split}
\end{equation}
where we used the derivative of the product in the last line.

Using expression \eqref{kluy} for $f = \zeta + (u - u \circ \sigma) + c$, with $\zeta \in \mathrm{Ker}(\mathcal{L}_f)$, we get for the derivative the equivalent expression
$$\frac{d}{dt}\int_\Omega f\, d \mu_{f + tg}|_{t=0}=\frac{d}{dt}\int_\Omega \zeta\, d \mu_{f + tg} |_{t=0} \;.$$

Now in \eqref{lind2}, taking $\varphi = \zeta$, from remark \ref{ptre}, we get
\begin{equation} \label{iddent4} 
\begin{split}
\frac{d}{dt}  h(\mu_{f + tg})|_{t=0} \,
&= -\,  \int_\Omega (I - \mathcal{L}_f)^{ -1} (\zeta) \,DN_f(g)\,  d\mu_f \\
&= -\,   \int_\Omega \zeta \,DN_f(g)\,  d \mu_f \;.
 \end{split}
 \end{equation}

We assumed that $g = \eta + (v - v \circ \sigma)$, where $\eta \in \mathrm{Ker}(\mathcal{L}_f)$, then,
\begin{equation} \label{ient423}  \frac{d}{dt}h( \mu_{f + tg}) =-\,   \int_\Omega \zeta \, DN_f(g)\,  d \mu_f = -\,   \int_\Omega \zeta \,\eta\,  d \mu_f \,. \end{equation}

\end{proof}

\medskip

\begin{corollary}  \label{cvaga1}  Given $f, g \in \mathrm{Lip}(\Omega)$, assume that  $f= \zeta + (u - u \circ \sigma) +c$ and $g = \eta + (v - v \circ \sigma)$, with $\zeta, \eta \in \mathrm{Ker}(\mathcal{L}_f)$ and   $\int_\Omega \eta^2 d \mu_f=1$. Then, the maximal value of 
\begin{equation} \label{lent1}  \frac{d}{dt}h( \mu_{f + tg})= -\,   \int_\Omega \zeta \,\eta\,  d \mu_f, \end{equation}
among all the $\eta$'s on the kernel of the operator $\mathcal{L}_f$,  indicates the direction of the larger increasing of entropy for $h( \mu_{f + tg})$, when $t$ is close to zero. 

\end{corollary}

Given $\varphi \in \mathrm{Lip}(\Omega)$, the {\em Pressure problem} about find the value $P(\varphi)$ (see the expression in \eqref{e1234}) for $T$-invariant probabilities $\mu$ is one of the main issues in Termodynamic Formalism (see \cite{PP} and \cite{LMMS}). For analyzing this problem a concept of entropy is required and this topic  was addressed via \eqref{ruop7777}. It is known the existence of a unique maximizing solution $\mu_B \in \mathcal{N}(\Omega)$, for $P(\varphi)$, where  $B \in \mathcal{N}(\Omega)$, and $\mu_B$ is called the equilibrium probability for $\varphi$ (and also for $B$). Even more, the functions $B$ and $\varphi$ are co-boundary.

 In the next theorem, we are interested in the directional derivatives associated with such a problem.
\smallskip

\begin{theorem} \label{Prepre}  Given $\varphi \in \mathrm{Lip}(\Omega)$, we are interested in estimating directional derivatives of the  pressure of $\varphi$ given by
\begin{equation} \label{e1234} P(\varphi) = \sup_{\mu \in \mathcal{M}_T(\Omega)} \Bigl\{ h(\mu) + \int_\Omega \varphi\, d \mu \Bigr\} \,.
\end{equation}
Denote by $\mu_B \in \mathcal{N}(\Omega)$, where $B \in \mathcal{N}(\Omega)$ the  unique maximizing solution for $P(\varphi)$ (see \cite{LMMS} for details). Consider another normalized potential $f$ and the associated equilibrium probability $\mu_f \in \mathcal{N}(\Omega)$, where $\mu_f \neq \mu_B$.

  Assume that $f= \zeta + (u - u \circ \sigma) +c$, $g = \eta + (v - v \circ \sigma)$ and $\varphi = \xi + (w - w \circ \sigma) +k$, where $\zeta, \eta, \xi \in \mathrm{Ker}(\mathcal{L}_f)$ and $c, k \in \mathbb{R}$. In this case, we do not assume that $g$ is in the Kernel of $\mathcal{L}_f$, but we assume for simplification  that it has integral $\int g d \mu_f=0$.

It is natural to consider variations of the law  $t \mapsto f + tg$ and the directional derivative
\begin{equation} \label{ie1237} \frac{d}{dt}\, \Bigl(\,h( \mu_{f + tg}) + \int_\Omega \varphi\,\, d \mu_{f + tg} \,\Bigr)|_{t=0}.\end{equation}
For fixed $\varphi$, we can ask about the maximal value of  \eqref{ie123} for  $\eta \in \mathcal{L}_f$, when $\eta$ satisfies $\int_\Omega \eta^2 = 1$.

  We show that  an  equivalent condition to  get \eqref{ie123} for a given  $g$ (and therefore for the associated $\eta \in \mathrm{Ker}(\mathcal{L}_f)$) is:
  \begin{equation} \label{vaga3} \int_\Omega \,\, (\xi\, - \zeta)\,\, \eta\, d \mu_f \;. \end{equation}
  
Then,  in the case \eqref{ie1237} is critical for all variations $\eta$ on the kernel (that is, $\eqref{ie1237}=0$), it will also be critical for all  variations $g$ not in the kernel (but with  $\mu_f$ mean zero).

\end{theorem}

\begin{proof} 
  From Corollary  \ref{vaga2} and  Proposition \ref{vaga1}, we have
  
\begin{equation} \label{lent17} 
\begin{split}
\frac{d}{dt}h(\mu_{f + tg})|_{t=0} + \frac{d}{dt} \int_\Omega \varphi \, d \mu_{f + tg}|_{t=0}
&= -\,   \int_\Omega \zeta \,\eta\,  d \mu_f + \int_\Omega \,\, \xi\, \eta\, d \mu_f \\ 
&= \int_\Omega \,\, (\xi\, - \zeta)\, \eta\, d \mu_f \;.   
 \end{split}
 \end{equation}

\end{proof}

Note that when $\mu_f$, which is Lipschitz-Gibbs probability for the normalized potential $f$, maximizes  \eqref{e1234}, the functions $f$ and $\varphi$ are co-boundaries. In this case $\xi\, - \zeta = 0$ and \eqref{ie123} is equal to zero in any direction $\eta$. Therefore, the probability $\mu_B$ is a critical point for the directional derivatives of the law $t \mapsto h( \mu_{f + tg}) + \int_\Omega \varphi\,\, d \mu_{f + tg}$.

\smallskip

In  Subsection \ref{subi} we  show  for some particular examples of $\mu_f$ a
more explicit  expression for \eqref{vaga}).

\medskip

Another expression for the second derivative of pressure can be obtained in the following way.

\begin{theorem} \label{ddef14}
Given $f, g \in \mathrm{Lip}(\Omega)$ and $t \in \mathbb{R}$, assume $f \in \mathcal{N}(\Omega)$ and denote by $p(t) := P(f + tg)$. Then,

\begin{equation} \label{lulu46}
 p^{\prime \prime}(0) = \frac{d^2}{dt^2} P(f + tg)|_{t=0} =\int_\Omega\, \Bigl(\, g - \int_\Omega g d \mu_f \,+ \,\frac{d}{dt}(\varphi_t - \varphi_t  \circ \sigma)|_{t=0} \,\Bigr)^2 \,d \mu_f \;,
\end{equation}
where $w_{f + tg} = e^{\varphi_t}$ is the eigenfunction for $f + tg$.
\end{theorem}

\begin{proof}
We assume that for all $t$ (small and close to $0$) and any $x$

$$\int_M J_t(ax) d\nu(a) = \int_M e^{ ( f(ax) + tg(ax) ) + [\varphi_t(ax) - \varphi_t(x)] - p(t)}d\nu(a) =1 \;.$$

where $J_t := e^{f + tg + \varphi_t - \varphi_t \circ \sigma - p(t)}$. We know that
\begin{equation}  \label{clo78} p'(0) = \frac{d}{dt} P(f+tg)|_{t=0} = \int_\Omega g d \mu_f, \end{equation}

\begin{equation}  \label{robep} \frac{d}{dt} \int_\Omega (\varphi_t - \varphi_t \circ \sigma) d\mu_f |_{t=0} = 0 \;,
\end{equation}
and
$$ \frac{d^2}{ dt^2}  \int_\Omega (\varphi_t - \varphi_t \circ \sigma) d\mu_f |_{t=0} = 0 \;.$$

Moreover,
\begin{align*}
0 =& \; \frac{d^2}{dt^2}  \int_\Omega \int_M e^{ ( f(ax) + tg(ax) ) + [\varphi_t(a x) - \varphi_t(x)] - p(t)} d\nu(a)  d\mu_f(x) \\
=& \; \frac{d}{dt} \int_\Omega \int_M J_t(ax)\,  \Bigl( g(ax) + \frac{d}{dt}(\varphi_t(a x) - \varphi_t (x)) - \frac{d}{dt} p(t) \Bigr)d\nu(a) d \mu_f(x) \\
=& \int_\Omega \int_M J_t(ax) \Bigl( g(ax) + \frac{d}{dt}(\varphi_t(ax) - \varphi_t(x)) - \frac{d}{dt} p(t)\Bigr)^2 \\
&+ J_t(ax) \Bigl( \frac{d^2}{ dt^2}(\varphi_t(ax) - \varphi_t(x)) - \frac{d^2}{dt^2} p(t) \Bigr)d\nu(a) d\mu_f(x) \;.
\end{align*}

 Then, from the above   and \eqref{robep}, it follows that
\begin{align*}
0 &= \frac{d^2}{dt^2}  \int_\Omega \int_M e^{ ( f(ax) + tg(ax) ) + [\varphi_t(a x) - \varphi_t(x)] - p(t)} d\nu(a)  d\mu_f(x)|_{t=0} \\
=& \int_\Omega \int_M \Bigl( g(ax) + \frac{d}{dt}(\varphi_t(ax) - \varphi_t(x))|_{t=0} - \frac{d}{dt} p(t)|_{t=0}\Bigr)^2 \\
&+ \Bigl( \frac{d^2}{ dt^2}(\varphi_t(ax) - \varphi_t(x))|_{t=0} - \frac{d^2}{dt^2} p(t)|_{t=0} \Bigr)d\nu(a) d\mu_f(x) \\ 
=& \int_\Omega \int_M \Bigl( g(ax) + \frac{d}{dt}(\varphi_t(ax) - \varphi_t(x))|_{t=0} - \frac{d}{dt} p(t)|_{t=0}\Bigr)^2 - \frac{d^2}{dt^2} p(t)|_{t=0} d\nu(a) d\mu_f(x) \;.
\end{align*}

   Finally, since we have $\mu_f \in \mathcal{M}_T(\Omega)$, it follows that

  $$ \frac{d^2}{dt^2} p(t)|_{t=0} = \int_\Omega \Bigl( g - \int_\Omega g d\mu_f + \frac{d}{dt}(\varphi_t - \varphi_t \circ \sigma)|_{t=0}\Bigr)^2 d\mu_f \;.$$
\end{proof}

\begin{theorem} \label{rafrug}
Suppose  $f \in \mathcal{N}(\Omega)$, we are going to take derivative on the direction $g \in \mathrm{Lip}(\Omega)$. Denote by $w(t,x) :=w_{f + tg}(x)$ the eigenfunction for $\mathcal{L}_{f + tg}$ associated to the eigenvalue $\lambda_{f + tg}$. Then, the function $x \mapsto \frac{\partial}{\partial t} w(t, x)|_{t=0}$ satisfies the following expression for all $x \in \Omega$
\begin{equation} \label{tripe1} (\mathcal{L}_f  - I) (\frac{\partial}{\partial t} w(t, .))|_{t=0} = \int_\Omega g d \mu_f -  \mathcal{L}_f (g) \;.
\end{equation}

Or, equivalently,
\begin{equation} \label{tripe112} \frac{\partial}{\partial t} w(t, .)|_{t=0} = (\mathcal{L}_f  - I)^{-1}  (\,\int_\Omega g d \mu_f -  \mathcal{L}_f (g) \,)\;.
\end{equation}

\end{theorem}
\begin{proof}

We will estimate for each $t \in \mathbb{R}$ the partial derivative of the eigenfunction $ \frac{\partial}{\partial t}  w(t, x)|_{t=0}$. Indeed, taking derivative on $t$, we obtain
$$ \frac{\partial}{\partial t}  \mathcal{L}_{f + tg} (w(t,.) ) (x) = \mathcal{L}_{f + tg} ( g\,w(t,.))(x) + \mathcal{L}_{f + tg} (\frac{\partial}{\partial t}   w(t,.)) (x). $$

On the other hand, for all $x \in \Omega$ we have
$$ \frac{\partial}{\partial t}( \lambda_{f + tg}\, w(t,x)\,)  = \, w(t,x) \frac{d}{dt}\lambda_{f + tg}\,+ \lambda_{f + tg}\,\frac{\partial}{\partial t}  w(t,x) \;.$$
Then,
\begin{equation} \label{uct} \mathcal{L}_{f + tg} ( g\,w(t,.))(x) + \mathcal{L}_{f + tg} (\frac{\partial}{\partial t}   w(t,.)) (x) =  w(t,x) \frac{d}{dt}\lambda_{f + tg}\,+ \lambda_{f + tg}\,\frac{\partial}{\partial t}  w(t,x) \;.
\end{equation}

Therefore, when $t=0$ we get
$$ \mathcal{L}_f (g)(x) + \mathcal{L}_f (\frac{\partial}{\partial t} w( t, .)|_{t=0})(x) = \int_\Omega g d \mu_f +\frac{\partial}{\partial t} w(t,x) |_{t=0} \;.$$

Finally,
\begin{equation} \label{uct1} (\mathcal{L}_f - I) (\frac{\partial}{\partial t}  w( t, .)|_{t=0} ) =  \int_\Omega g d \mu_f -  \mathcal{L}_f(g) \;.
\end{equation}
\end{proof}

The next result was relevant in \cite{LR1}.

\begin{theorem} \label{saKL} Under the former hypothesis, if $\eta \in \mathrm{Ker}(\mathcal{L}_f)$ and  $w(t,x) :=w_{f + t\, \eta}(x)$. Then, $\frac{\partial}{\partial t}  w(t, x) |_{t=0}$ is a constant function independent of $x$.
\end{theorem}

\begin{proof} Since we assume $\eta \in \mathrm{Ker}(\mathcal{L}_f)$, it follows that $\int_\Omega \eta \, d \mu_f = 0$. Therefore, from \eqref{uct1} we get that
$$ (\mathcal{L}_f - I) (\frac{\partial}{\partial t} w(t,.)|_{t=0} ) = 0 \;.$$

The above, in particular implies that $(\frac{\partial}{\partial t} w(t,x)|_{t=0} )$ is constant.
\end{proof}

\subsection{Examples with explicit computation} \label{subi}

We consider in this section the symbolic space $\Omega := \{0,1\}^\mathbb{N}$ and the shift $T$ acting on it; this corresponds to take the space  of symbols $M=\{0,1\}.$

In this subsection, we want to show that for some particular examples of $\mu$ one can
get more explicit expressions which are helpful in applications of Theorem \ref{Prepre} (see expression \eqref{vaga333}). In order to do that we take advantage of the existence of an orthonormal basis in the Kernel of a certain Ruelle operator (see \cite{LR1} and \cite{CHLS}).

\medskip

We denote by $P$ the line stochastic matrix (with all entries positive)
 \begin{equation} \label{tororo2} P= \left(
\begin{array}{cc}
P_{0,0} & P_{0,1}\\
P_{1,0} & P_{1,1}
\end{array}\right).
\end{equation}

Consider the invariant Markov probability $\mu$ obtained from the stochastic matrix $(P_{i,j})_{i,j=0,1}$ and the unique initial {\bf left invariant {\em vector of probability} $\pi=(\pi_0,\pi_1)\in [0,1]^2$.}

The associated normalized potential $B=\log J$ satisfies: $J:\{0,1\}^\mathbb{N} \to (0,1)$
is constant equal
$$J_{i,j}=\frac{\pi_i \,P_{i,j}}{\pi_j}$$
on the {\bf cylinder $[ij] := \{x \in \Omega | x_1 = i; x_2 = j \}$, $i,j=0,1$.}

Let $\omega := (\omega_1, \omega_2,..., \omega_n) \in \{0, 1\}^n$ be a finite word on the symbols $\{0,1\}$. Denote by $[\omega]$ the cylinder set $[\omega] := \{x \in \Omega | x_1=\omega_1, ... , x_n=\omega_n  \}$. The empty word, denoted by $\emptyset$, is also considered a finite word.

 Given $\mu$ as above, consider the Hilbert space $L^2 (\mu)$ equipped with the inner product $<f,g>= \int_\Omega f \, g \,d \mu.$

The family of H\"older functions
\begin{align}
   e_{\omega} := \frac{1}{\sqrt{\mu([\omega])}} \sqrt{\frac{P_{\omega_n,1}}{P_{\omega_n,0}\,}} \, {\bf 1}_{[\omega0]} - \frac{1}{\sqrt{\mu([\omega])}} \sqrt{\frac{P_{\omega_n,0}}{P_{\omega_n,1} }} \, {\bf 1}_{[\omega1]},
    \label{eq52}
    \end{align}
$\omega := (\omega_1, \omega_2,..., \omega_n) \in \{0, 1\}^n$, is an orthonormal set for $L^2 (\mu)$ (see \cite{KS} for a general expression  and \cite{LR1}   for \eqref{eq52}; see also \cite{CHLS}). In order to get a (Haar) basis we have to  add
$e_{\emptyset} := \frac{1}{ \sqrt{\mu ([0])}}  {\bf 1}_{[0]} - \frac{1}{ \sqrt{\mu ([1])}} {\bf 1}_{[1]}$ to this family.

Given a finite word $\omega := (\omega_1, \omega_2, ... , \omega_n) \in \{0, 1\}^n$, we denote
\begin{equation*}  \label{luc1} a_\omega := \frac{ \sqrt{\pi_{\omega_1}  }} {\sqrt{\pi_{0}}\sqrt{P_{0,\omega_1} }} \,\, e_{0\omega} - \,\frac{ \sqrt{\pi_{\omega_1}  }} {\sqrt{\pi_{1}}\sqrt{P_{1,\omega_1} }}\,\,e_{1\omega},
\end{equation*}

Finally, we set
\begin{equation} \label{pede} \hat{a}_\omega : =\frac{1}{\|a_\omega\|}\, a_\omega \,.
\end{equation}
the normalization of $a_\omega$.

\smallskip

In order to get a complete orthonomal set for the kernel of the Ruelle operator $\mathcal{L}_{\log J}$ acting on $L^2 (\mu)$, we will have to add to the functions of the form \eqref{pede}  two more functions:   $\hat{a}_{\emptyset}^0$ and $\hat{a}_{\emptyset}^1$         , which are properly described in Definition 6.8 in \cite{LR1}. 

 Denote by $\{0, 1\}^*$ the set of all the finite words on $\{0, 1\}$, i.e., $\{0, 1\}^* := \emptyset \cup \bigcup_{n=1}^\infty \{0, 1\}^n$. One of the main results in \cite{LR1} is

\begin{theorem} \label{ort-basis}
 The  family $(\hat{a}_\omega)_{\omega \in \{0, 1\}^*}$, plus the two functions $\hat{a}_{\emptyset}^0$ and $\hat{a}_{\emptyset}^1$, determine an orthonormal set on the kernel of the Ruelle operator $\mathcal{L}_{\log J}$ acting on $L^2 (\mu)$.
\end{theorem}

The set of elements in this orthonormal  family will be generically described by $(\hat{a}_\omega)_\omega$ (we allow $\omega$ to be the empty word $\emptyset$.)

An item of major interest  here is to express the Lipschitz function $\eta \in \mathcal{L}_{\log J}$ on the form
$$\eta =\sum_{\omega \in \{0, 1\}^*}\, \eta_\omega \hat{a}_\omega, $$
where $\eta_\omega := <\eta , \hat{a}_\omega>.$

An important issue in expression \eqref{lent17} in  Theorem \ref{Prepre} was to express $f= \zeta + (u - u \circ \sigma) +c$, $g = \eta + (v - v \circ \sigma)$ and $\varphi = \xi + (w - w \circ \sigma) +k$, where $\zeta, \eta, \xi \in \mathrm{Ker}(\mathcal{L}_f)$ and $c, k \in \mathbb{R}$. Here we take $g=\eta$, and  $f=\log J$, where $J$ was defined above. 

Therefore, given $\varphi$ we set
$$\zeta= \sum_ {\omega \in \{0, 1\}^*}\, f_\omega \hat{a}_\omega \;,$$
and 
$$\xi =\sum_ {\omega \in \{0, 1\}^*}\, \varphi_\omega \hat{a}_\omega  \;.$$

Since the family  $(\hat{a}_\omega)_\omega$ is orthonormal on $\mathrm{Ker}(\mathcal{L}_{\log J})$, it follows that

\begin{equation} \label{vaga333} 
\begin{split}
 \int_\Omega \,\, (\xi\, - \zeta)\,\, \eta\, d \mu 
 &= \int_\Omega  \Bigl(\,\sum_ {\omega \in \{0, 1\}^*}\, (\varphi_\omega - f_\omega)\hat{a}_\omega\,   \sum_ {\omega \in \{0, 1\}^*}\, \eta_\omega \hat{a}_\omega \,\Bigr)d \mu \\
 &= \,\sum_ {\omega \in \{0, 1\}^*}\, (\varphi_\omega - f_\omega)\, \eta_\omega.
\end{split}    
\end{equation}

\section*{Acknowledgements}

We would like to thank to FCT - Project UIDP/00144/2020 - (Portugal) and CNPq - (Brazil) for financial support during the development of this paper.

We thank R. Ruggiero for helpful comments related to  the topic of the paper.

\end{document}